\documentclass[11pt,reqno]{amsart}
\usepackage[foot]{amsaddr}
\usepackage{amsthm,amssymb,amsfonts}
\usepackage{amsmath}
\usepackage{amscd} 
\usepackage{microtype}

\usepackage[normalem]{ulem}
\usepackage[all]{xy}
\usepackage{setup}

\usetikzlibrary{shadows, decorations.pathmorphing, decorations.markings, calc, shapes}

\renewcommand{\P}{\mathbb{P}}

\usepackage{bbm}
\renewcommand{\1}{\mathbf{1}}

\newcommand{\ad}{\text{ and }}

\newcommand{\ceq}{\coloneqq} 

\newcommand{\e}{\varepsilon}

\newcommand{\bs}{\baselineskip}

\newcommand{\ind}[1]{\1_{[#1]}} 

\newcommand{\E}[1]{{\mathbf E}\left[#1\right]}				
\newcommand{\va}{{\mathbf{Var}}}

\newcommand{\cat}{\mathrm{Cat}}
\newcommand{\dtcs}{\mathrm{DTCS}}
\newcommand{\bal}{\mathrm{CBT}}

\title{Extremal subtrees of critical beta-splitting trees}

\author{Anna Brandenberger$^*$ \and Byron Chin$^\dagger$ \and Elchanan Mossel$^*$}

\address{$^*$Department of Mathematics, MIT}
\email{\{abrande,elmos\}@mit.edu}

\address{$^\dagger$School of Mathematics, Georgia Institute of Technology}
\email{bchin@gatech.edu}

\begin{document}

\begin{abstract}
We determine the most and least likely shapes for an instance of the critical beta-splitting tree via a connection to data compression and Huffman's minimum redundancy codes. This allows us to answer combinatorial questions about the distribution of clades posed by Aldous and Janson, stated as problem 7 in \cite{aldous2023critical}. 
\end{abstract}

\maketitle

\section{Introduction}

We study the distribution of \emph{clades} in critical beta-splitting trees. A clade, or \emph{fringe subtree}, of a rooted tree is a subtree induced by an internal vertex together with all of its descendants. Originally introduced by Aldous~\cite{aldous1996probability} as a model for phylogenetic trees, the critical beta-splitting tree has been the subject of renewed mathematical interest due to its qualitatively different behaviour from other models of random trees; see~\cite{aldous2023critical} for a comprehensive survey, and~\cite{aldous2024critical3,aldous2024critical4,aldous2024harmonic,aldous2025critical,brandenberger2025asymptoticsharmonicdescentchain,iksanov2025harmonic,kolesnik2025critical} for the recent series of papers on the model.

\subsection{The model}\label{subsec: defs}
For any $m \geq 2$, let the distribution $(q_{m}(i))_{i=1}^{m-1}$ be
\begin{equation}\label{eq:qmi}
    q_m(i) = \frac{m}{2h_{m-1}} \frac{1}{i(m-i)},
\end{equation}
where $h_{m-1} = \sum_{j=1}^{m-1} 1/j$ is the harmonic sum. The \emph{discrete-time critical beta-splitting tree} $\dtcs(n)$ is a random binary tree with $n$ leaves (labelled $[n]$) built as follows. First, split $[n]$ into left and right subtrees $\{1, \dots, L_n\}$ and $\{L_n + 1, \dots, n\}$ of respective sizes $L_n$ and $R_n \ceq n-L_n$, where $L_n$ has distribution $q_n(\cdot)$.
Then, recursively split each interval of size $m \geq 2$ into two subtrees according to $q_m(\cdot)$, stopping when $m=1$. The name $\dtcs$ distinguishes the model from the continuous-time version $\mathrm{CTCS}(n)$ of Aldous and Pittel~\cite{aldous2025critical}, in which an interval of size $m$ waits an $\mathrm{Exp}(h_{m-1})$ time before splitting according to $q_m(\cdot)$; the two models have the same tree shape.

The general $\beta$-splitting model of~\cite{aldous1996probability} is defined analogously, with $q_m(\cdot)$ replaced by the probability vector
\begin{equation}\label{eq:qbeta}
  q^\beta_m(i) \; \propto \; \frac{\Gamma(i+\beta+1)\,\Gamma(m-i+\beta+1)}{\Gamma(i+1)\,\Gamma(m-i+1)},
  \qquad 1 \le i \le m-1,
\end{equation}
which for $\beta = -1$ reduces to~\eqref{eq:qmi} and in general behaves like $(i(m-i))^{\beta}$ for $1 \ll i \ll m$. The value $\beta = -1$ is deemed \emph{critical} because of a discontinuity in the order of magnitude of leaf heights: they are of order $n^{-\beta-1}$ for $-2<\beta<-1$ and of order $\log n$ for $\beta > -1$, whereas at $\beta=-1$ a typical leaf height is of order $\log^2 n$; see~\cite{aldous1996probability,aldous2025critical,aldous2024critical4}.

\subsection{Notation} For a tree $T$, we define its size to be its number of leaves. For two rooted binary trees $T$ and $T'$ we say $T = T'$ if they are identical as rooted ordered trees (plane trees), and we say $T \cong T'$ if they are identical as unordered trees.
For example, writing $[T_1\,T_2]$ for the rooted ordered tree whose left and right subtrees are $T_1$ and $T_2$, we have $[[1 \, 2]\,3] \cong [3\, [1 \, 2]]$, but they are not identical as ordered trees. Throughout, $\log$ denotes the natural logarithm and $\log_2$ denotes the base-two logarithm.

\subsection{Main results}
We consider questions about clades of potentially growing size asked by Aldous and Janson in \cite[Open problem 7]{aldous2023critical}. 
For a realization of $\dtcs(n)$, write $N_n(\chi)$ for the number of clades of $\dtcs(n)$ whose shape is $\chi$, and let
\[
  K_n \ceq \sum_{\chi \text{ ordered}} \ind{N_n(\chi) \ge 1},
  \qquad
  K_n^u \ceq \sum_{\chi \text{ unordered}} \ind{N_n(\chi) \ge 1}
\]
be the numbers of distinct ordered and unordered clade shapes that occur.
\cite[Open problem 7]{aldous2023critical} 
asks for the asymptotics of
\begin{enumerate}
    \item the number $K_n$ of distinct clade shapes occurring in $\dtcs(n)$,
    \item the (size of the) largest clade that appears more than once in $\dtcs(n)$,
    \item and the (size of the) smallest clade that does not appear in $\dtcs(n)$.
\end{enumerate}
We make a connection between the critical beta-splitting distribution and Huffman's minimum-redundancy codes and use it to identify the extremal ordered shapes. 
By determining the most and least likely shapes for a given clade size, we answer the questions above in Theorems~\ref{thm:number}, \ref{thm:max} and \ref{thm:min} respectively, up to constant factors.

\begin{thm}\label{thm:min}
    The size of both the smallest ordered and the smallest unordered tree that do not appear is $(1 + o(1)) \frac{\log n}{\log\log\log n}$ with high probability as $n\to\infty$.
\end{thm}

\begin{thm}\label{thm:max}
    The size of the largest ordered tree that appears at least twice is in $[(1/\alpha - o(1))\log_2 n, (2/\alpha+o(1))\log_2 n]$ and the size of the largest unordered tree that appears at least twice is in $[(1/\alpha-o(1))\log_2 n, (2/(\alpha-1)+o(1))\log_2 n]$ with high probability, where $\alpha = 1 + \sum_{i=1}^\infty \frac{{\log_2} h_{2^i-1}}{2^i} = 1.637...$.
\end{thm}

\begin{thm}\label{thm:number}
    The numbers  of {distinct ordered and unordered subtree shapes $K_n$ and $K_n^u$ are both} $\Theta\left(\frac{n\log\log n}{\log n}\right)$ with high probability.
\end{thm}

\subsection{Related work}\label{subsec:related}
Fringe trees have been studied extensively in other models of random trees. Aldous~\cite{A:91} carried out a general study covering many classes of models; more recent work includes analyses for families of branching processes~\cite{cai2017study,HJ:17,janson2016asymptotic} as well as for Patricia tries and compressed binary search trees~\cite{janson2026fringetreespatriciatries}. Janson's work~\cite{janson2026fringetreespatriciatries} also proves central limit theorems for the beta-splitting trees with parameters $\beta = -3/2, 0, \infty$, and proves convergence in probability
for the critical parameter $\beta = -1$; the full central limit theorem for the number of clades of any fixed size or shape was obtained in prior work of the authors~\cite{brandenberger2025asymptoticsharmonicdescentchain}. 

All three of our main theorems have well-studied analogues for other tree models, where the motivation comes from DAG compression of trees: the number of distinct fringe subtrees is the size of the minimal DAG representing the tree. 
Flajolet, Sipala and Steyaert~\cite{flajolet1990analytic} showed that a uniformly random binary (or plane) tree with $n$ nodes has $\Theta(n/\sqrt{\log n})$ distinct fringe subtrees. 
On the other hand, random binary search trees have $\Theta(n/\log n)$ many, as shown by Flajolet, Gourdon and Mart{\'\i}nez~\cite{flajolet1997patterns} 
and Devroye~\cite{devroye1998richness}. This is an analogue of our Theorem~\ref{thm:number}, whose scaling interestingly differs only by a $\log \log n$ factor. Devroye's work also determines the largest $K$ such that \emph{every} shape of size at most $K$ occurs in a random binary search tree, and is thus the direct analogue of our Theorem~\ref{thm:min}; the analogue of Theorem~\ref{thm:max} for simply generated trees is due to Ralaivaosaona and Wagner~\cite{ralaivaosaona2015repeated}.
Sharp results in this direction, including the constant in the $\Theta(n/\log n)$ asymptotics for binary search trees, were obtained by Seelbach Benkner and Wagner~\cite{seelbach2020collection,seelbach2022distinct,wagner2024distinct}.

\subsection*{Acknowledgements}
 A.B.\ is supported by NSERC PGS-D. A.B.\ and B.C.\ were previously supported by NSF GRFP 2141064. E.M.\ is partially supported by ARO MURI N00014241274, by Vannevar Bush Faculty Fellowship ONR-N00014-20-1-2826 and by a Simons Investigator Award. All ideas, proofs and main text were generated and written by the authors. AI was used only for light proofreading and copy-editing.

\section{Analysis of subtree distributions}\label{sec:subtree-distributions}

\subsection{Ordered subtrees}
Let $p_{\max}(k)$ and $p_{\min}(k)$ respectively denote the probability of the most and least likely ordered binary trees of size $k$. Conditioning on the two sizes in the root split and using the independence of the two descendant trees gives the recursions
\[
p_{\max}(k) = \max_{1 \leq i\leq k-1} q_k(i) p_{\max}(i) p_{\max}(k-i) = \frac{k}{2h_{k-1}} \max_{1 \leq i\leq k-1} \frac{p_{\max}(i)}{i} \frac{p_{\max}(k-i)}{k-i}
\]
and analogously for $p_{\min}$ with $p_{\min}(1) = p_{\max}(1) = 1$. Letting $a_{\min} (k) \ceq p_{\min}(k)/k$ and $a_{\max}(k) \ceq p_{\max}(k)/k$,
we can work with 
\begin{align}
    a_{\max}(k) &= \frac{1}{2h_{k-1}} \max_{1 \leq i \leq k-1} a_{\max}(i)a_{\max}(k-i), \ad \label{eq:amax-def}\\
    a_{\min}(k) &= \frac{1}{2h_{k-1}} \min_{1 \leq i \leq k-1} a_{\min}(i)a_{\min}(k-i). \label{eq:amin-def}
\end{align}
Notice that \eqref{eq:amax-def} and \eqref{eq:amin-def} correspond to placing a weight of $\tfrac{1}{2h_{s_v - 1}}$ at each internal vertex $v$ with size $s_v$ (i.e., $s_v$ leaves in its subtree) and finding a binary tree with $k$ leaves that respectively maximizes or minimizes the total product. In particular, the most/least likely trees correspond respectively to trees $T$ that minimize/maximize 
\begin{equation}\label{eq:tree-cost}
    G(T) \ceq \sum_{v \in I(T)} g(s_v)
\end{equation}
where $g(x) \ceq \log(h_{x-1})$ for $x > 1$ and $I(T)$ is the set of internal vertices of $T$. Observe that $g(x)$ in \eqref{eq:tree-cost} is non-decreasing and strictly concave. Consequently, these optimization problems map directly onto a classic data compression setting from coding theory~\cite{huffman1952method}. Assign fixed, non-negative weights $(w_u)_{u \in L(T)}$ to the leaves of a tree and set internal vertex weights $w_v$ to be the sum of the leaf weights in the subtree rooted at $v$. Our setting corresponds to assigning a weight of 1 for each leaf. This implies that the order of trees in terms of likelihood is exactly the order of trees in terms of cost with respect to $g$. 

We use this relationship to identify that the most likely and least likely shapes are respectively the complete binary tree and the caterpillar tree (also referred to as the comb or totally pectinate tree). We formally define these below; see also Figure~\ref{fig:cbt-and-cat} for an illustration.

\begin{figure}[hbtp]
    \centering
    \begin{tikzpicture}[x=0.75pt,y=0.75pt,yscale=-1,xscale=1]

\draw    (190,30) -- (210,50) ;
\draw    (210,50) -- (230,70) ;
\draw    (250,90) -- (270,110) ;
\draw    (270,110) -- (290,130) ;
\draw [shift={(290,130)}, rotate = 45] [color={rgb, 255:red, 0; green, 0; blue, 0 }  ][fill={rgb, 255:red, 0; green, 0; blue, 0 }  ][line width=0.75]      (0, 0) circle [x radius= 2.01, y radius= 2.01]   ;
\draw    (210,50) -- (190,70) ;
\draw [shift={(190,70)}, rotate = 135] [color={rgb, 255:red, 0; green, 0; blue, 0 }  ][fill={rgb, 255:red, 0; green, 0; blue, 0 }  ][line width=0.75]      (0, 0) circle [x radius= 2.01, y radius= 2.01]   ;
\draw    (230,70) -- (210,90) ;
\draw [shift={(210,90)}, rotate = 135] [color={rgb, 255:red, 0; green, 0; blue, 0 }  ][fill={rgb, 255:red, 0; green, 0; blue, 0 }  ][line width=0.75]      (0, 0) circle [x radius= 2.01, y radius= 2.01]   ;
\draw    (250,90) -- (230,110) ;
\draw [shift={(230,110)}, rotate = 135] [color={rgb, 255:red, 0; green, 0; blue, 0 }  ][fill={rgb, 255:red, 0; green, 0; blue, 0 }  ][line width=0.75]      (0, 0) circle [x radius= 2.01, y radius= 2.01]   ;
\draw    (270,110) -- (250,130) ;
\draw [shift={(250,130)}, rotate = 135] [color={rgb, 255:red, 0; green, 0; blue, 0 }  ][fill={rgb, 255:red, 0; green, 0; blue, 0 }  ][line width=0.75]      (0, 0) circle [x radius= 2.01, y radius= 2.01]   ;
\draw    (190,30) -- (170,50) ;
\draw [shift={(170,50)}, rotate = 135] [color={rgb, 255:red, 0; green, 0; blue, 0 }  ][fill={rgb, 255:red, 0; green, 0; blue, 0 }  ][line width=0.75]      (0, 0) circle [x radius= 2.01, y radius= 2.01]   ;
\draw    (230,70) -- (250,90) ;
\draw    (80,50) -- (40,70) ;
\draw    (40,70) -- (60,90) ;
\draw    (40,70) -- (20,90) ;
\draw    (60,90) -- (70,110) ;
\draw [shift={(70,110)}, rotate = 63.43] [color={rgb, 255:red, 0; green, 0; blue, 0 }  ][fill={rgb, 255:red, 0; green, 0; blue, 0 }  ][line width=0.75]      (0, 0) circle [x radius= 2.01, y radius= 2.01]   ;
\draw    (60,90) -- (50,110) ;
\draw [shift={(50,110)}, rotate = 116.57] [color={rgb, 255:red, 0; green, 0; blue, 0 }  ][fill={rgb, 255:red, 0; green, 0; blue, 0 }  ][line width=0.75]      (0, 0) circle [x radius= 2.01, y radius= 2.01]   ;
\draw    (20,90) -- (30,110) ;
\draw [shift={(30,110)}, rotate = 63.43] [color={rgb, 255:red, 0; green, 0; blue, 0 }  ][fill={rgb, 255:red, 0; green, 0; blue, 0 }  ][line width=0.75]      (0, 0) circle [x radius= 2.01, y radius= 2.01]   ;
\draw    (20,90) -- (10,110) ;
\draw [shift={(10,110)}, rotate = 116.57] [color={rgb, 255:red, 0; green, 0; blue, 0 }  ][fill={rgb, 255:red, 0; green, 0; blue, 0 }  ][line width=0.75]      (0, 0) circle [x radius= 2.01, y radius= 2.01]   ;
\draw    (80,50) -- (120,70) ;
\draw    (120,70) -- (140,90) ;
\draw [shift={(140,90)}, rotate = 45] [color={rgb, 255:red, 0; green, 0; blue, 0 }  ][fill={rgb, 255:red, 0; green, 0; blue, 0 }  ][line width=0.75]      (0, 0) circle [x radius= 2.01, y radius= 2.01]   ;
\draw    (120,70) -- (100,90) ;
\draw [shift={(100,90)}, rotate = 135] [color={rgb, 255:red, 0; green, 0; blue, 0 }  ][fill={rgb, 255:red, 0; green, 0; blue, 0 }  ][line width=0.75]      (0, 0) circle [x radius= 2.01, y radius= 2.01]   ;

\end{tikzpicture}
    \caption{Complete binary tree $\bal(6)$ (left) and Caterpillar tree $\cat(6)$ (right).}
    \label{fig:cbt-and-cat}
\end{figure}
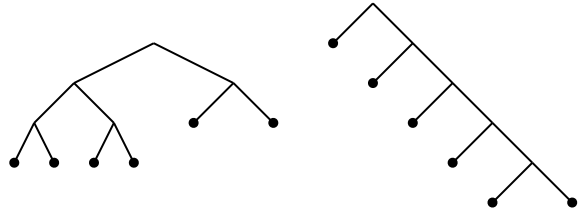

\begin{defn}
    Let $\bal(k)$ be the complete binary tree with $k$ leaves, that is, where every level is completely full except possibly the last, which is filled from the left.
    See Figure~\ref{fig:cbt-and-cat}~(left) for a diagram. 
\end{defn}

\begin{defn}
    Let $\cat(k)$ be the rooted ordered binary tree of size $k$ where all splits are of size 1. See Figure~\ref{fig:cbt-and-cat}~(right) for a diagram.
    Note that $\cat(k)$ has $2^{k-2}$ orderings which occur with equal probability.
\end{defn}

\begin{lem}[Most likely ordered trees]\label{lem:max}
    Out of all rooted ordered binary trees of size $k$, the trees occurring with maximal probability are orderings of $\bal(k)$. Moreover, 
    \[ \P(\dtcs(k) = \bal(k)) = 2^{-(\alpha + o(1)) k} \]
    where $\alpha = 1 + \sum_{i=1}^\infty \frac{\log_2 h_{2^i-1}}{2^i} = 1.637...$, {as $k\to\infty$}.
\end{lem}
\begin{proof}
    Glassey and Karp \cite[Theorem 1]{glassey1976optimality} prove that the Huffman tree constructed from any fixed leaf weights $(w_u)$ minimizes not only the standard cost $\sum_{v \in I(T)} w_v$ (which governs the expected depth of the leaves), but also the general cost $\sum_{v \in I(T)} g(w_v)$ for any non-decreasing, concave function $g$.
    Our setting, where $w_v$ equals the number of descendant leaves $s_v$, corresponds exactly to assigning uniform leaf weights $w_u = 1$ for all $u \in L(T)$. Therefore, the trees that minimize $G(T)$ are simply the Huffman trees for $k$ uniform weights. This is well-known (see, e.g., \cite[Section 4]{glassey1976optimality}) to be exactly the complete binary tree $\bal(k)$, up to ordering.

    To compute the probability, we start by assuming $k = 2^h$. At height $h-i$ there are $2^{h-i}$ internal vertices whose descendant subtrees have size $2^i$, so multiplying their split probabilities from \eqref{eq:qmi} gives
    \begin{align*}
        \log_2 \P(\dtcs(k) = \bal(k)) &= \log_2\prod_{i=1}^h \left(\frac{2^i}{2\cdot (2^{i-1})^2 \cdot h_{2^i-1}}\right)^{2^{h-i}} \\
        &= \sum_{i=1}^h 2^{h-i}\log_2\left(\frac{1}{2^{i-1} \cdot h_{2^i-1}}\right) \\
        &= - 2^h \left(\sum_{i=1}^h \frac{i-1}{2^i} + \frac{\log_2 h_{2^i-1}}{2^i}\right) \\
        &= -(\alpha+o(1))k.
    \end{align*}
    Now let $k$ be arbitrary and take $h=\lfloor 2\log_2\log k\rfloor$. After removing the top $h$ levels of $\bal(k)$, all components are $\bal(2^d)$ or $\bal(2^{d+1})$, except possibly one complete tree $E$ of size $s\leq 2^{d+1}$, where $d=\lfloor\log_2 k\rfloor-h$. Let $N$ and $M$ be the respective numbers of the two power-of-two components, and put $s=0$ if there is no exceptional component. Thus, as $s=O(k/\log^2 k)$,
    \[
        N2^d+M2^{d+1}=k-s=k-o(k).
    \]
    Every split probability in a tree of size at most $k$ is at least $k^{-2}$. There are fewer than $2^h$ splits in the removed levels and at most $s-1$ in $E$, so, writing $P_r=\P(\dtcs(r)=\bal(r))$, we obtain
    \[
        k^{-2(2^h+s)}P_{2^d}^{N}P_{2^{d+1}}^{M}
        \leq \P(\dtcs(k)=\bal(k))
        \leq P_{2^d}^{N}P_{2^{d+1}}^{M}.
    \]
    Here $2^h\log_2 k=o(k)$ and $s\log_2 k=o(k)$, while the power-of-two calculation gives
    \[
        P_{2^d}^{N}P_{2^{d+1}}^{M}
        =2^{-(\alpha+o(1))(N2^d+M2^{d+1})}
        =2^{-(\alpha+o(1))k}.
    \]
    The sandwiching proves the claimed exponent.
\end{proof}

Although \cite{glassey1976optimality} do not consider the cost maximization problem, their same perspective and method can be used. We identify the least likely shape as the caterpillar tree, which is the most imbalanced and is known to be extremal in various other settings as well. We use the name caterpillar in connection with its unordered counterpart defined later (see Definition~\ref{def:unlabel-cat}).

\begin{lem}[Least likely ordered trees]\label{lem:min ordered}
    Out of all ordered binary trees of size $k$, the trees occurring with minimal probability are orderings of $\cat(k)$, with 
    \[
    \P(\dtcs(k) = \cat(k)) = k2^{-k+1}\prod_{i=2}^k \frac{1}{h_{i-1}} = \exp(-(1+o(1))k\log\log k),\quad k\to\infty.
    \]
\end{lem}
\begin{proof} 
To show that \eqref{eq:tree-cost} is maximized for $\cat(k)$, we use a majorization theory fact (see \cite[I.3.C.1]{marshall2011inequalities}, also used/reproven in \cite[Lemma 2]{glassey1976optimality}) that it suffices to show that the vector $(s_v)_{v \in I(\cat(k))}$ is weakly supermajorized by $(s_v)_{v \in I(T)}$ for any other binary tree $T$. That is, letting $s_{(1)} \leq \cdots \leq s_{(k-1)}$ denote the components of $(s_v)$ in increasing order, it suffices to show that 
\begin{equation}\label{eq:cat-tree-fact}
        \sum_{1 \leq i \leq \ell} s_{(i)}(\cat(k)) \geq \sum_{1 \leq i \leq \ell} s_{(i)}(T)
\end{equation}
for all $1 \leq \ell \leq k-1$. In fact, we show that $s_{(i)}(\cat(k)) \geq s_{(i)}(T)$ for every $1 \leq i \leq k-1$, so \eqref{eq:cat-tree-fact} comes from a stronger coordinate-wise domination. For any $T$ and any $m \geq 2$, let $v$ be the internal vertex of minimal size satisfying $s_v \geq m$. It must have $s_v-2$ descendants of size at most $m-1$. If $s_v > m$ this is already at least $m-1$ vertices, and if $s_v = m$ we have at least $m-1$ vertices of size at most $m$ by counting $v$ and its descendants. Thus, $s_{(m-1)}(T) \leq m = s_{(m-1)}(\cat(k))$.
\end{proof}

\subsection{Unordered subtrees}
Let $p_{\min}^u(k)$  denote the probability of the least likely unordered binary tree of size $k$. This can be computed recursively as 
\[
p_{\min}^u(k) = \min_{1 \leq i \leq k-1} 2^{\ind{i \neq k/2}} q_k(i) p_{\min}^u(i) p_{\min}^u(k-i) = \min_{1 \leq i\leq k-1} \frac{k}{2^{\ind{i=k/2}}h_{k-1}}\frac{p_{\min}^u(i)}{i} \frac{p_{\min}^u(k-i)}{k-i}.
\]
Letting $a_{\min}^u (k) \ceq p_{\min}^u(k)/k$,
we can work with 
\begin{equation}\label{eq:recursion-unlabel}
    a_{\min}^u(k) =  \min_{1 \leq i \leq k-1} \frac{1}{2^{\ind{i=k/2}}h_{k-1}} a_{\min}^u(i)a_{\min}^u(k-i).
\end{equation}
This additional term of $2^{\ind{i=k/2}}$ capturing the orderings complicates the situation greatly. Notice that the least likely ordered tree has many uneven splits, and thus has many ($2^{k-2}$) orderings; whereas the most likely ordered tree is quite balanced, and thus has few orderings. This suggests that the ordering may have a significant impact on the distribution of trees. Heuristically, $p_{\min}(k)$ is small enough compared to the number of orderings to not be affected much; whereas $p_{\max}(k)$ is not large enough to overcome the lack of orderings. Consequently, we are able to describe the least likely unordered tree, but the structure of the most likely unordered tree appears to be erratic. Indeed, simulations show that it resembles $\bal(k)$, but with many even splits replaced by splits of small difference.

While we cannot describe the most likely unordered shape, we can give exponential upper and lower bounds on its probability.
\begin{lem}[Most likely unordered probability]\label{lem:max unordered}
    For $\alpha = 1 + \sum_{i=1}^\infty \frac{\log_2 h_{2^i-1}}{2^i} = 1.637...$, 
    \[ 2^{-(\alpha + o(1))k} \leq \max_T \P(\dtcs(k) \cong T) \leq 2^{-(\alpha-1+o(1))k}. \]
\end{lem}
\begin{proof}
    The bounds follow from
    \[ \P(\dtcs(k) = T) \leq \P(\dtcs(k) \cong T) \leq \P(\dtcs(k)=T) \cdot \mathrm{Ord}(T) \]
    combined with Lemma~\ref{lem:max} and $\mathrm{Ord}(T) \leq 2^k$ where $\mathrm{Ord}(T) = {2^{k-1}}/{|\mathrm{Aut}(T)|}$ is the number of orderings of $T$.
\end{proof}
Note that the same argument shows that the least likely unordered probability is $\exp(-(1+o(1))k\log\log k)$, which suffices to prove the main results. However, with a bit more work we can describe the shape that achieves this minimum. The shape resembles $\cat(k)$ from the ordered case, but has a small deviation due to the ordering factor for small $k$.

\begin{defn}\label{def:unlabel-cat}
    For $k \leq 16$, define $\cat^u(k)$ to be the trees in Figure~\ref{fig:mini cat}. 
    For $k \geq 17$, define $\cat^u(k)$ to be the binary tree of size $k$ where the first $k-16$ splits are of size 1 and the last 16 leaves form a $\cat^u(16)$ = $\bal(16)$, the complete binary tree of size 16. See Figure~\ref{fig:cat} for a diagram. 
\end{defn}
\begin{figure}[ht]
    \centering
    \begin{tikzpicture}[x=0.75pt,y=0.75pt,yscale=-1,xscale=1]

\draw    (80,50) -- (90,70) ;
\draw [shift={(90,70)}, rotate = 63.43] [color={rgb, 255:red, 0; green, 0; blue, 0 }  ][fill={rgb, 255:red, 0; green, 0; blue, 0 }  ][line width=0.75]      (0, 0) circle [x radius= 2.01, y radius= 2.01]   ;
\draw    (80,50) -- (70,70) ;
\draw [shift={(70,70)}, rotate = 116.57] [color={rgb, 255:red, 0; green, 0; blue, 0 }  ][fill={rgb, 255:red, 0; green, 0; blue, 0 }  ][line width=0.75]      (0, 0) circle [x radius= 2.01, y radius= 2.01]   ;
\draw    (130,50) -- (140,70) ;
\draw [shift={(140,70)}, rotate = 63.43] [color={rgb, 255:red, 0; green, 0; blue, 0 }  ][fill={rgb, 255:red, 0; green, 0; blue, 0 }  ][line width=0.75]      (0, 0) circle [x radius= 2.01, y radius= 2.01]   ;
\draw    (140,30) -- (120,70) ;
\draw [shift={(120,70)}, rotate = 116.57] [color={rgb, 255:red, 0; green, 0; blue, 0 }  ][fill={rgb, 255:red, 0; green, 0; blue, 0 }  ][line width=0.75]      (0, 0) circle [x radius= 2.01, y radius= 2.01]   ;
\draw    (140,30) -- (160,70) ;
\draw [shift={(160,70)}, rotate = 63.43] [color={rgb, 255:red, 0; green, 0; blue, 0 }  ][fill={rgb, 255:red, 0; green, 0; blue, 0 }  ][line width=0.75]      (0, 0) circle [x radius= 2.01, y radius= 2.01]   ;
\draw    (220,30) -- (240,50) ;
\draw    (220,30) -- (200,50) ;
\draw    (240,50) -- (250,70) ;
\draw [shift={(250,70)}, rotate = 63.43] [color={rgb, 255:red, 0; green, 0; blue, 0 }  ][fill={rgb, 255:red, 0; green, 0; blue, 0 }  ][line width=0.75]      (0, 0) circle [x radius= 2.01, y radius= 2.01]   ;
\draw    (240,50) -- (230,70) ;
\draw [shift={(230,70)}, rotate = 116.57] [color={rgb, 255:red, 0; green, 0; blue, 0 }  ][fill={rgb, 255:red, 0; green, 0; blue, 0 }  ][line width=0.75]      (0, 0) circle [x radius= 2.01, y radius= 2.01]   ;
\draw    (200,50) -- (210,70) ;
\draw [shift={(210,70)}, rotate = 63.43] [color={rgb, 255:red, 0; green, 0; blue, 0 }  ][fill={rgb, 255:red, 0; green, 0; blue, 0 }  ][line width=0.75]      (0, 0) circle [x radius= 2.01, y radius= 2.01]   ;
\draw    (200,50) -- (190,70) ;
\draw [shift={(190,70)}, rotate = 116.57] [color={rgb, 255:red, 0; green, 0; blue, 0 }  ][fill={rgb, 255:red, 0; green, 0; blue, 0 }  ][line width=0.75]      (0, 0) circle [x radius= 2.01, y radius= 2.01]   ;
\draw    (320,30) -- (310,40) ;
\draw  [line width=0.75]  (310,40) -- (340,70) -- (280,70) -- cycle ;
\draw    (320,30) -- (360,70) ;
\draw [shift={(360,70)}, rotate = 45] [color={rgb, 255:red, 0; green, 0; blue, 0 }  ][fill={rgb, 255:red, 0; green, 0; blue, 0 }  ][line width=0.75]      (0, 0) circle [x radius= 2.01, y radius= 2.01]   ;
\draw    (340,10) -- (400,70) ;
\draw [shift={(400,70)}, rotate = 45] [color={rgb, 255:red, 0; green, 0; blue, 0 }  ][fill={rgb, 255:red, 0; green, 0; blue, 0 }  ][line width=0.75]      (0, 0) circle [x radius= 2.01, y radius= 2.01]   ;
\draw    (340,10) -- (335,15) ;
\draw    (325,25) -- (320,30) ;
\draw    (90,130) -- (130,150) ;
\draw    (130,150) -- (150,170) ;
\draw    (130,150) -- (110,170) ;
\draw    (90,130) -- (50,150) ;
\draw    (150,170) -- (160,190) ;
\draw [shift={(160,190)}, rotate = 63.43] [color={rgb, 255:red, 0; green, 0; blue, 0 }  ][fill={rgb, 255:red, 0; green, 0; blue, 0 }  ][line width=0.75]      (0, 0) circle [x radius= 2.01, y radius= 2.01]   ;
\draw    (150,170) -- (140,190) ;
\draw [shift={(140,190)}, rotate = 116.57] [color={rgb, 255:red, 0; green, 0; blue, 0 }  ][fill={rgb, 255:red, 0; green, 0; blue, 0 }  ][line width=0.75]      (0, 0) circle [x radius= 2.01, y radius= 2.01]   ;
\draw    (110,170) -- (120,190) ;
\draw [shift={(120,190)}, rotate = 63.43] [color={rgb, 255:red, 0; green, 0; blue, 0 }  ][fill={rgb, 255:red, 0; green, 0; blue, 0 }  ][line width=0.75]      (0, 0) circle [x radius= 2.01, y radius= 2.01]   ;
\draw    (110,170) -- (100,190) ;
\draw [shift={(100,190)}, rotate = 116.57] [color={rgb, 255:red, 0; green, 0; blue, 0 }  ][fill={rgb, 255:red, 0; green, 0; blue, 0 }  ][line width=0.75]      (0, 0) circle [x radius= 2.01, y radius= 2.01]   ;
\draw    (50,150) -- (70,170) ;
\draw    (50,150) -- (30,170) ;
\draw    (70,170) -- (80,190) ;
\draw [shift={(80,190)}, rotate = 63.43] [color={rgb, 255:red, 0; green, 0; blue, 0 }  ][fill={rgb, 255:red, 0; green, 0; blue, 0 }  ][line width=0.75]      (0, 0) circle [x radius= 2.01, y radius= 2.01]   ;
\draw    (70,170) -- (60,190) ;
\draw [shift={(60,190)}, rotate = 116.57] [color={rgb, 255:red, 0; green, 0; blue, 0 }  ][fill={rgb, 255:red, 0; green, 0; blue, 0 }  ][line width=0.75]      (0, 0) circle [x radius= 2.01, y radius= 2.01]   ;
\draw    (30,170) -- (40,190) ;
\draw [shift={(40,190)}, rotate = 63.43] [color={rgb, 255:red, 0; green, 0; blue, 0 }  ][fill={rgb, 255:red, 0; green, 0; blue, 0 }  ][line width=0.75]      (0, 0) circle [x radius= 2.01, y radius= 2.01]   ;
\draw    (30,170) -- (20,190) ;
\draw [shift={(20,190)}, rotate = 116.57] [color={rgb, 255:red, 0; green, 0; blue, 0 }  ][fill={rgb, 255:red, 0; green, 0; blue, 0 }  ][line width=0.75]      (0, 0) circle [x radius= 2.01, y radius= 2.01]   ;
\draw    (240,140) -- (230,150) ;
\draw  [line width=0.75]  (230,150) -- (270,190) -- (190,190) -- cycle ;
\draw    (240,140) -- (290,190) ;
\draw [shift={(290,190)}, rotate = 45] [color={rgb, 255:red, 0; green, 0; blue, 0 }  ][fill={rgb, 255:red, 0; green, 0; blue, 0 }  ][line width=0.75]      (0, 0) circle [x radius= 2.01, y radius= 2.01]   ;
\draw    (260,120) -- (330,190) ;
\draw [shift={(330,190)}, rotate = 45] [color={rgb, 255:red, 0; green, 0; blue, 0 }  ][fill={rgb, 255:red, 0; green, 0; blue, 0 }  ][line width=0.75]      (0, 0) circle [x radius= 2.01, y radius= 2.01]   ;
\draw    (260,120) -- (255,125) ;
\draw    (245,135) -- (240,140) ;
\draw  [line width=0.75]  (410,140) -- (460,190) -- (360,190) -- cycle ;

\draw (80.5,89.5) node  [font=\small] [align=left] {$\displaystyle k=2$};
\draw (140.5,90.5) node  [font=\small] [align=left] {$\displaystyle k=3$};
\draw (220.5,90.5) node  [font=\small] [align=left] {$\displaystyle k=4:\mathrm{CBT}( 4)$};
\draw (290,56.4) node [anchor=north west][inner sep=0.75pt]  [font=\scriptsize,color={rgb, 255:red, 0; green, 0; blue, 0 }  ,opacity=1 ]  {$\mathrm{CBT}( 4)$};
\draw (314,90.5) node  [font=\small] [align=left] {$\displaystyle k=5,6,7$};
\draw (380,70.5) node  [font=\footnotesize] [align=left] {$\displaystyle \dotsc $};
\draw (330.2,19.2) node  [font=\tiny,rotate=-315] [align=left] {$\displaystyle ...$};
\draw (90.5,210.5) node  [font=\small] [align=left] {$\displaystyle k=8:\mathrm{CBT}( 8)$};
\draw (235.5,210.5) node  [font=\small] [align=left] {$\displaystyle k=9,10,\dotsc ,15$};
\draw (208,173.4) node [anchor=north west][inner sep=0.75pt]  [font=\footnotesize,color={rgb, 255:red, 0; green, 0; blue, 0 }  ,opacity=1 ]  {$\mathrm{CBT}( 8)$};
\draw (310,190.5) node  [font=\footnotesize] [align=left] {$\displaystyle \dotsc $};
\draw (250.2,129.2) node  [font=\tiny,rotate=-315] [align=left] {$\displaystyle ...$};
\draw (355,72) node [anchor=north west][inner sep=0.75pt]  [color={rgb, 255:red, 150; green, 150; blue, 150 }  ,opacity=1 ] [align=left] {$\displaystyle \underbrace{\ \ \ \ \ \ \ \ \ \ }_{k-4}$};
\draw (285,191.2) node [anchor=north west][inner sep=0.75pt]  [color={rgb, 255:red, 150; green, 150; blue, 150 }  ,opacity=1 ] [align=left] {$\displaystyle \underbrace{\ \ \ \ \ \ \ \ \ \ }_{k-8}$};
\draw (384,168.4) node [anchor=north west][inner sep=0.75pt]  [font=\footnotesize,color={rgb, 255:red, 0; green, 0; blue, 0 }  ,opacity=1 ]  {$\mathrm{CBT}( 16)$};
\draw (410.5,210) node  [font=\small] [align=left] {$\displaystyle k=16$};

\end{tikzpicture}
    \caption{Unordered caterpillar trees of size $k=2, \ldots, 16$.}
    \label{fig:mini cat}
\end{figure}
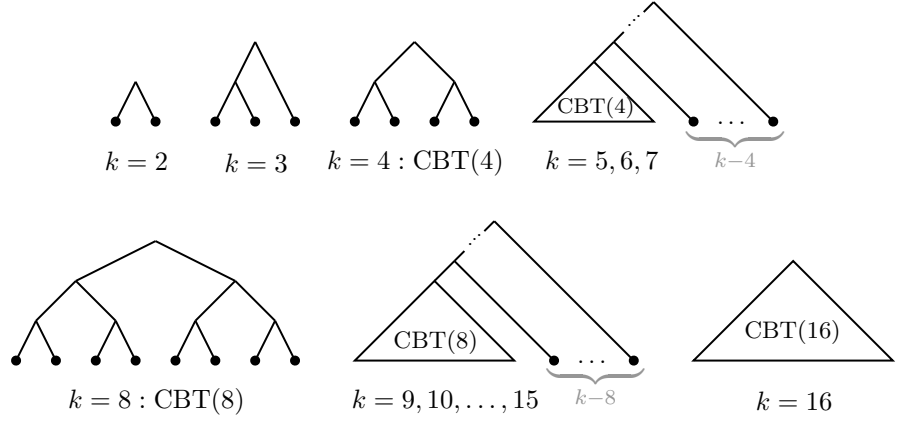

\begin{figure}[ht]
  \centering
  \begin{tikzpicture}[x=0.75pt,y=0.75pt,yscale=-1,xscale=1]

\draw    (90,40) -- (80,50) ;
\draw    (90,40) -- (150,100) ;
\draw [shift={(150,100)}, rotate = 45] [color={rgb, 255:red, 0; green, 0; blue, 0 }  ][fill={rgb, 255:red, 0; green, 0; blue, 0 }  ][line width=0.75]      (0, 0) circle [x radius= 2.01, y radius= 2.01]   ;
\draw    (110,20) -- (190,100) ;
\draw [shift={(190,100)}, rotate = 45] [color={rgb, 255:red, 0; green, 0; blue, 0 }  ][fill={rgb, 255:red, 0; green, 0; blue, 0 }  ][line width=0.75]      (0, 0) circle [x radius= 2.01, y radius= 2.01]   ;
\draw    (110,20) -- (105,25) ;
\draw    (95,35) -- (90,40) ;
\draw  [line width=0.75]  (80,50) -- (130,100) -- (30,100) -- cycle ;

\draw (170,100.5) node  [font=\footnotesize] [align=left] {$\displaystyle \dotsc $};
\draw (100.2,29.2) node  [font=\tiny,rotate=-315] [align=left] {$\displaystyle ...$};
\draw (144,102) node [anchor=north west][inner sep=0.75pt]  [color={rgb, 255:red, 150; green, 150; blue, 150 }  ,opacity=1 ] [align=left] {$\displaystyle \underbrace{\ \ \ \ \ \ \ \ \ \ }_{k-16}$};
\draw (55,78.4) node [anchor=north west][inner sep=0.75pt]  [font=\footnotesize,color={rgb, 255:red, 0; green, 0; blue, 0 }  ,opacity=1 ]  {$\mathrm{CBT}( 16)$};

\end{tikzpicture}
  \vspace{-.25\bs}
  \caption{Unordered caterpillar tree of size $k \geq 17$: a length $k-16$ spine with a hanging $\bal(16)$.}
  \label{fig:cat}
\end{figure}
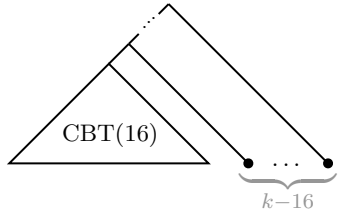

\begin{lem}[Least likely unordered tree]\label{lem:min unordered}
    Out of all unordered binary trees of size $k$, the tree with the minimum probability is $\cat^u(k)$. Moreover, for $k \geq 17$ we have 
    \[ \P(\dtcs(k) \cong \cat^u(k)) = \frac{kp_{\min}^u(16)}{16}\cdot \prod_{i=17}^k \frac{1}{h_{i-1}} = \exp(-(1+o(1))k\log\log k) {,\quad k\to\infty}. \]
\end{lem}
\begin{proof} 
We aim to show that $a_{\min}^u(k) = a_{\min}^u(16)\cdot\prod_{i=17}^k 1/h_{i-1}$ and that, for $1\leq k\leq16$, the values $a_{\min}^u(k)$ are attained by the trees in Definition~\ref{def:unlabel-cat}. Finite verification for $k \le 35$ can be performed with exact rational arithmetic: starting from $a_{\min}^u(1)=1$, for each $2\leq k\leq35$ we evaluate every entry on the right-hand side of \eqref{eq:recursion-unlabel}. The minimizing smaller root split is $k/2$ for $k\in\{4,8,16\}$ and is $1$ for every other $k\leq35$. This supplies the base case for the induction below.

Take $k \geq {36}$ and assume we have shown the desired conclusion for $j \leq k-1$. By \eqref{eq:recursion-unlabel} and symmetry, we would like to show that the minimizer of
\[ \frac{1}{2^{\ind{{j}=k/2}}h_{k-1}} a_{\min}^u(j)a_{\min}^u(k-j) \]
over $1 \leq j \leq k/2$ occurs at $j=1$. We separate into three cases: $2\leq j \leq 16$, $17 \leq j < k/2$, and $j = k/2$. 

First, if $17 \leq j < k/2$ then by the inductive hypothesis we have 
\begin{align*}
    a_{\min}^u(j)a_{\min}^u(k-j) = a_{\min}^u(16)\cdot\prod_{i=17}^j \frac{1}{h_{i-1}} \cdot a_{\min}^u(16)\cdot\prod_{i=17}^{k-j} \frac{1}{h_{i-1}}.
\end{align*}
For $17\leq j<(k-1)/2$, the ratio of the expression at $j+1$ to that at $j$ is
$h_{k-j-1}/h_j>1$. Hence the expression is increasing in $j$ and is minimized at $j=17$. Thus, it suffices to show that
\begin{align*}
    \frac{a_{\min}^u(17)a_{\min}^u(k-17)}{a_{\min}^u(1)a_{\min}^u(k-1)} &= \frac{a_{\min}^u(16)\cdot\prod_{i=17}^{17} \frac{1}{h_{i-1}} \cdot a_{\min}^u(16)\cdot\prod_{i=17}^{k-17} \frac{1}{h_{i-1}}}{a_{\min}^u(16) \cdot \prod_{i=17}^{k-1}\frac{1}{h_{i-1}}} \\ 
    &= a_{\min}^u(16) \cdot \frac{1}{h_{16}} \cdot \prod_{i=k-16}^{k-1} h_{i-1} > 1.
\end{align*}
Notice that the product on the right hand side is increasing in $k$. Since the inequality holds at $k=35$ by the base-case calculation, it holds for all larger $k$ as well. 

Next, if $j=k/2$ then we need to compare $a_{\min}^u(1)a_{\min}^u(k-1)$ with $a_{\min}^u(k/2)a_{\min}^u(k/2)/2$. By the inductive hypothesis, 
\begin{align*}
    \frac{a_{\min}^u(k/2)a_{\min}^u(k/2)/2}{a_{\min}^u(1)a_{\min}^u(k-1)} &= \frac{a_{\min}^u(16)\cdot\prod_{i=17}^{k/2} \frac{1}{h_{i-1}} \cdot a_{\min}^u(16)\cdot\prod_{i=17}^{k/2} \frac{1}{h_{i-1}}}{2a_{\min}^u(16)\cdot\prod_{i=17}^{k-1} \frac{1}{h_{i-1}}} \\
    &= \frac{a_{\min}^u(16)}{2} \cdot \frac{\prod_{i=k/2+1}^{k-1} h_{i-1}}{\prod_{i=17}^{k/2} h_{i-1}}.
\end{align*}
When we increment $k/2$ to $k/2+1$, the right hand side is multiplied by $h_{k-1}h_{k}/h_{k/2}^2 > 1$. Thus, {since the inequality holds at $k=34$ by the base-case calculation,} it holds for all larger even $k$ as well. 

Finally, we have the case of $2 \leq j \leq 16$. By the inductive hypothesis we have 
\begin{align*}
    \frac{a_{\min}^u(j)a_{\min}^u(k-j)}{a_{\min}^u(1)a_{\min}^u(k-1)} &= \frac{a_{\min}^u(j)\cdot a_{\min}^u(16)\cdot \prod_{i=17}^{k-j}\frac{1}{h_{i-1}}}{a_{\min}^u(16)\cdot \prod_{i=17}^{k-1}\frac{1}{h_{i-1}}} \\
    &= a_{\min}^u(j) \cdot \prod_{i=k-j+1}^{k-1} h_{i-1}.
\end{align*}
For each $2 \leq j \leq 16$, as $k$ increases the product on the right hand side increases. {By the base-case calculation, the right hand side is greater than 1 at $k=35$ for all such $j$,} so it holds for all larger $k$ as well. 
\end{proof}

\section{Proof of main results}
Throughout this section, $n$ is sufficiently large for the displayed ranges of $k$.
To begin, we need an estimate on how many subtrees of size $k$ exist in $\dtcs(n)$.
\begin{lem}
\label{lem:trials}
Let $2\leq k=n^{o(1)}$. With high probability, the number of subtrees of size $k$ in $\dtcs(n)$ is $\Theta(n\log k/k^2)$. Moreover, conditional on their locations, the shapes of these subtrees are independent and distributed as $\dtcs(k)$. There exists $C > 0$ such that for $2\leq k = n^{o(1)}$,
\[
    \P\left(\left|N_n(k)-\E{N_n(k)}\right|>\tfrac12\E{N_n(k)}\right)
    \leq C\frac{k^4}{n(\log k)^2}.
\]
\end{lem}
\begin{proof}[Proof of Lemma~\ref{lem:trials}]
By \cite[Lemmas 4.1, 4.2]{brandenberger2025asymptoticsharmonicdescentchain}, for $k = n^{o(1)}$ there exist $\mu(k)$ and $\sigma(k)$ such that $\E{N_n(k)} = \Theta(\mu(k) n)$ and $\va(N_n(k)) \leq \sigma(k)^2n$. The explicit identification of the limit in \cite{aldous2025critical} yields $\mu(k) = \frac{6h_{k-1}}{\pi^2 (k-1)^2} = \Theta(\frac{\log k}{k^2})$ and when combined with the proof of Lemma 4.2 in \cite{brandenberger2025asymptoticsharmonicdescentchain} yields $\sigma(k)^2$ is at most a universal constant. The stated probability bound now follows from Chebyshev's inequality, and the independence follows directly from the recursive construction of the tree.
\end{proof}

We now combine Lemma~\ref{lem:trials} with our knowledge of the extremal shapes to deduce the main results.
\begin{proof}[Proof of Theorem~\ref{thm:min}]
Let $k_0=\log n/\log\log\log n$ and fix $\e>0$. By Lemma~\ref{lem:trials}, the probability that
\[
    N_n(k)\geq c\frac{n\log k}{k^2}
\]
fails for some $2\leq k\leq(1-\e)k_0$ is at most
\[
    \frac{C}{n}\sum_{k\leq k_0}\frac{k^4}{(\log k)^2}
    \leq\frac{Ck_0^5}{n}=o(1).
\]
Let $\mathcal E_n$ be the event that these bounds hold simultaneously, and let $\chi$ be an ordered shape of size $k\leq(1-\e)k_0$. Lemma~\ref{lem:min ordered} gives
\[
    \P(N_n(\chi)=0,\mathcal E_n)
    \leq\exp\left(-c\frac{n\log k}{k^2}p_{\min}(k)\right)
    \leq\exp(-n^{\e/2}),
\]
because, uniformly for $k\leq(1-\e)k_0$ tending to infinity,
\[
    \log\left(c\frac{n\log k}{k^2}p_{\min}(k)\right)
    =\log n-(1+o(1))k\log\log k\geq\frac\e2\log n.
\]
The finitely many smaller sizes have absence probability exponentially small in $n$. There are fewer than $4^k$ shapes of size $k$, so a union bound over shapes and over $2\leq k\leq(1-\e)k_0$ proves the lower bound.

For $k=\lceil(1+\e)k_0\rceil$, Lemmas~\ref{lem:trials} and~\ref{lem:min ordered} give
\[
    \log\E{N_n(\cat(k))}
    \leq\log n-(1+o(1))k\log\log k
    \leq-\frac\e2\log n.
\]
Markov's inequality proves the upper bound. The unordered statement follows in the same way from Lemma~\ref{lem:min unordered}; its probability has the same asymptotic form $\exp(-(1+o(1))k\log\log k)$.
\end{proof}

\begin{proof}[Proof of Theorem~\ref{thm:max}]
First consider ordered trees and set $k_0=\alpha^{-1}\log_2n$. Fix $\e>0$ and let $k=\lfloor(1-\e)k_0\rfloor$. By Lemma~\ref{lem:trials}, except on an event of probability
\[
    O\left(\frac{k^4}{n(\log k)^2}\right)=o(1),
\]
there are at least $M=\lfloor n/k^2\rfloor$ clades of size $k$. By Lemma~\ref{lem:max}, $p_{\max}(k)=n^{-1+\e+o(1)}$, so $Mp_{\max}(k)=n^{\e+o(1)}/k^2\to\infty$. Conditional on these $M$ clades, the probability of at most one copy of $\bal(k)$ is at most
\[
    (1-p_{\max}(k))^M+Mp_{\max}(k)(1-p_{\max}(k))^{M-1}
    \leq(1+Mp_{\max}(k))e^{-(M-1)p_{\max}(k)}
    \leq \exp(-n^{\e/2})
\]
for all sufficiently large $n$. This proves the lower bound.

Now let $k>(2+\e)k_0$. For $k\leq4k_0$, Lemma~\ref{lem:max} gives
$n^2p_{\max}(k)\leq n^{-\e/2}$ for all sufficiently large $n$. The probability that there are no repeats among the size-$k$ clades is at least
\[
    \prod_{j=1}^{n-1}(1-jp_{\max}(k))
    \geq\exp(-n^2p_{\max}(k))
    \geq1-n^{-\e/2},
\]
where the middle inequality uses $jp_{\max}(k)\leq1/2$ and
$\log(1-x)\geq-2x$ for $0\leq x\leq1/2$. A union bound over $(2+\e)k_0<k\leq4k_0$ is therefore $o(1)$. For $k>4k_0$, we have $n^2p_{\max}(k)\leq n^{-2+o(1)}$, so another union bound over $k<n$ proves the ordered upper bound.

For unordered trees, the lower-bound argument is unchanged because
$\P(\dtcs(k)\cong\bal(k))\geq\P(\dtcs(k)=\bal(k))$. For the upper bound, repeat the preceding argument with $k_0=(\alpha-1)^{-1}\log_2n$ and Lemma~\ref{lem:max unordered}.
\end{proof}

\begin{proof}[Proof of Theorem~\ref{thm:number}]
Let $L=\log_2n$. The probability that the lower bound in Lemma~\ref{lem:trials} fails for some $4L\leq k\leq5L$ is at most
\[
    \frac{C}{n}\sum_{k=\lceil4L\rceil}^{\lfloor5L\rfloor}
    \frac{k^4}{(\log k)^2}=o(1).
\]
Thus, with high probability, the number of clades in this range is at least
\[
    c\sum_{k=\lceil4L\rceil}^{\lfloor5L\rfloor}\frac{n\log k}{k^2}
    \geq c'\frac{n\log\log n}{\log n}.
\]
By Theorem~\ref{thm:max}, all of them are distinct as ordered and unordered subtrees with high probability. This proves the lower bound.

For the upper bound, we consider a few ranges of $k$. We use $c$ and $C$ as universal constants that may change from line to line.
\begin{itemize}
    \item $k\leq\frac{L}{3}$: The total number of distinct binary trees of size $k$ is at most $4^k$, so the contribution from this case is at most
    \[
        \sum_{k=2}^{{\lfloor L/3\rfloor}}4^k\leq4^{L/3+1}\leq4n^{2/3}.
    \]
    \item $\frac{L}{3}\leq k\leq L^{12}$: The probability that the upper bound in Lemma~\ref{lem:trials} fails for some $k$ in this range is at most
    \[
        \frac{C}{n}\sum_{k\leq L^{12}}\frac{k^4}{(\log k)^2}
        \leq\frac{CL^{60}}{n}=o(1).
    \]
    Outside this event, the number of subtrees in the range is at most
    \[
        C\sum_{k={\lceil L/3\rceil}}^{\lfloor L^{12} \rfloor}
        \frac{n\log k}{k^2}
        \leq C\frac{n\log\log n}{\log n}.
    \]
    \item ${L^{12}} \leq k\leq n-n^{1/6}$: Let $x_m=\E{N_m(k)}/m$ and $r=\lfloor k^{1/6}\rfloor$. The case $c=3$, $\ell=1$ of \cite[Theorem~1.1(iii)]{brandenberger2025asymptoticsharmonicdescentchain} gives
    \[
        x_{k+r}\leq\frac{3x_k}{rh_k}\leq6k^{-7/6}.
    \]
    Here $x_k=1/k$, $r\geq k^{1/6}/2$, and $h_k\geq1$. By part (i) of the same theorem, $x_m$ decreases in $m$. Since $n-k\geq n^{1/6}\geq r$, we have $x_n\leq6k^{-7/6}$. Hence, with $A=\lceil L^{12} \rceil$,
    \[
        \E{\sum_{k=A}^{\lfloor n-n^{1/6}\rfloor}N_n(k)}
        \leq6n\sum_{k\geq A}k^{-7/6}
        \leq50nA^{-1/6}
        \leq\frac{50n}{\log^2n}.
    \]
    By Markov's inequality, the probability of more than $n/\log n$ such subtrees is at most $50/\log n$.
    \item $n-n^{1/6}\leq k\leq n$: There is at most one subtree of each of these sizes, so the contribution is at most $n^{1/6} +1$.
\end{itemize}
Summing these cases yields the desired bound.
\end{proof}

\bibliographystyle{abbrv}
\bibliography{biblio}

@InProceedings{aldous1996probability,
author="Aldous, David",
editor="Aldous, David
and Pemantle, Robin",
title="Probability Distributions on Cladograms",
booktitle="Random Discrete Structures",
year="1996",
publisher="Springer",
address="New York, NY",
series="The IMA Volumes in Mathematics and its Applications",
volume="76",
pages="1--18",
isbn="978-1-4612-0719-1"
}

@article{aldous2023critical,
  title={The critical beta-splitting random tree {II}: Overview and open problems},
  author={Aldous, David and Janson, Svante},
  journal={arXiv preprint arXiv:2303.02529},
  year={2023}
}

@article{aldous2025critical,
  title={The critical beta-splitting random tree {I}: Heights and related results},
  author={Aldous, David and Pittel, Boris},
  journal={The Annals of Applied Probability},
  volume={35},
  number={1},
  pages={158--195},
  year={2025},
  publisher={Institute of Mathematical Statistics}
}

@article{aldous2024harmonic,
  title={The harmonic descent chain},
  author={Aldous, David and Janson, Svante and Li, Xiaodan},
  journal={Electronic Communications in Probability},
  volume={29},
  pages={1--10},
  year={2024},
  publisher={The Institute of Mathematical Statistics and the Bernoulli Society}
}

@article{iksanov2025harmonic,
  title={The harmonic descent chain and regenerative composition structures},
  author={Iksanov, Alexander},
  journal={Electronic Communications in Probability},
  volume={30},
  pages={1--3},
  year={2025},
  publisher={The Institute of Mathematical Statistics and the Bernoulli Society}
}

@article{aldous2024critical3,
  title={The critical beta-splitting random tree {III}: The exchangeable partition representation and the fringe tree},
  author={Aldous, David and Janson, Svante},
  journal={arXiv preprint arXiv:2412.09655},
  year={2024}
}

@article{aldous2024critical4,
    author = {David Aldous and Svante Janson},
    title = {{The critical beta-splitting random tree IV: Mellin analysis of leaf height}},
    volume = {30},
    journal = {Electronic Journal of Probability},
    publisher = {Institute of Mathematical Statistics and Bernoulli Society},
    pages = {1 -- 39},
    year = {2025},
    doi = {10.1214/25-EJP1332},
    URL = {https://doi.org/10.1214/25-EJP1332}
}

@article{kolesnik2025critical,
  title={Critical beta-splitting, via contraction},
  author={Kolesnik, Brett},
  journal={Electronic Communications in Probability},
  volume={30},
  pages={1--14},
  year={2025},
  publisher={The Institute of Mathematical Statistics and the Bernoulli Society}
}

@article{brandenberger2025asymptoticsharmonicdescentchain,
      title={Asymptotics for the harmonic descent chain and applications to critical beta-splitting trees}, 
      author={Anna Brandenberger and Byron Chin and Elchanan Mossel},
      year={2025},
      journal={arXiv preprint arXiv:2505.24821},
}

@article{huffman1952method,
    author={Huffman, David A.},
  journal={Proceedings of the IRE}, 
  title={A Method for the Construction of Minimum-Redundancy Codes}, 
  year={1952},
  volume={40},
  number={9},
  pages={1098-1101},
  doi={10.1109/JRPROC.1952.273898}
}

@article{glassey1976optimality,
  title={On the optimality of {H}uffman trees},
  author={Glassey, CR and Karp, RM},
  journal={SIAM Journal on Applied Mathematics},
  volume={31},
  number={2},
  pages={368--378},
  year={1976},
  publisher={SIAM}
}

@book{marshall2011inequalities,
  title={Inequalities: Theory of Majorization and Its Applications},
  author={Marshall, Albert W. and Olkin, Ingram and Arnold, Barry C.},
  edition={2nd},
  year={2011},
  publisher={Springer},
  address={New York, NY},
  series={Springer Series in Statistics},
  doi={10.1007/978-0-387-68276-1},
  isbn={978-0-387-40087-7},
  url={https://link.springer.com/book/10.1007/978-0-387-68276-1}
}

@article{janson2026fringetreespatriciatries,
      title={Fringe trees of {P}atricia tries, compressed binary search trees, and three other random full binary trees}, 
      author={Svante Janson},
      year={2024},
      journal={arXiv preprint arXiv:2405.01239},
}

@article {HJ:17,
    AUTHOR = {Holmgren, Cecilia and Janson, Svante},
     TITLE = {Fringe trees, {C}rump--{M}ode--{J}agers branching processes and
              {$m$}-ary search trees},
   JOURNAL = {Probab. Surv.},
  FJOURNAL = {Probability Surveys},
    VOLUME = {14},
      YEAR = {2017},
     PAGES = {53--154},
      ISSN = {1549-5787},
   MRCLASS = {60C05 (05C05 05C80 60J80 60J85 68P05 68P10)},
  MRNUMBER = {3626585},
MRREVIEWER = {Nicolas\ Broutin},
       DOI = {10.1214/16-PS272},
       URL = {https://doi.org/10.1214/16-PS272},
}

@article {A:91,
    AUTHOR = {Aldous, David},
     TITLE = {Asymptotic fringe distributions for general families of random
              trees},
   JOURNAL = {Ann. Appl. Probab.},
  FJOURNAL = {The Annals of Applied Probability},
    VOLUME = {1},
      YEAR = {1991},
    NUMBER = {2},
     PAGES = {228--266},
      ISSN = {1050-5164,2168-8737},
   MRCLASS = {60C05 (05C80)},
  MRNUMBER = {1102319},
MRREVIEWER = {Jos\'e\ L.\ Palacios},
       URL =
              {http://links.jstor.org/sici?sici=1050-5164(199105)1:2<228:AFDFGF>2.0.CO;2-8&origin=MSN},
}

@article{janson2016asymptotic,
author = {Janson, Svante},
title = {Asymptotic normality of fringe subtrees and additive functionals in conditioned {G}alton--{W}atson trees},
journal = {Random Structures \& Algorithms},
volume = {48},
number = {1},
pages = {57-101},
doi = {https://doi.org/10.1002/rsa.20568},
url = {https://onlinelibrary.wiley.com/doi/abs/10.1002/rsa.20568},
eprint = {https://onlinelibrary.wiley.com/doi/pdf/10.1002/rsa.20568},
year = {2016}
}

@article {cai2017study,
    AUTHOR = {Cai, Xing Shi and Devroye, Luc},
     TITLE = {A study of large fringe and non-fringe subtrees in conditional
              {G}alton--{W}atson trees},
   JOURNAL = {ALEA Lat. Am. J. Probab. Math. Stat.},
  FJOURNAL = {ALEA. Latin American Journal of Probability and Mathematical
              Statistics},
    VOLUME = {14},
      YEAR = {2017},
    NUMBER = {1},
     PAGES = {579--611},
      ISSN = {1980-0436},
   MRCLASS = {60C05 (60J80)},
  MRNUMBER = {3667924},
MRREVIEWER = {Tatyana\ S.\ Turova},
       DOI = {10.30757/alea.v14-29},
       URL = {https://doi.org/10.30757/alea.v14-29},
}

@inproceedings{flajolet1990analytic,
  title={Analytic variations on the common subexpression problem},
  author={Flajolet, Philippe and Sipala, Paolo and Steyaert, Jean-Marc},
  booktitle={Automata, Languages and Programming (ICALP 1990)},
  series={Lecture Notes in Computer Science},
  volume={443},
  pages={220--234},
  year={1990},
  publisher={Springer}
}

@article{flajolet1997patterns,
  title={Patterns in random binary search trees},
  author={Flajolet, Philippe and Gourdon, Xavier and Mart{\'\i}nez, Conrado},
  journal={Random Structures \& Algorithms},
  volume={11},
  number={3},
  pages={223--244},
  year={1997},
  publisher={Wiley}
}

@article{devroye1998richness,
  title={On the richness of the collection of subtrees in random binary search trees},
  author={Devroye, Luc},
  journal={Information Processing Letters},
  volume={65},
  number={4},
  pages={195--199},
  year={1998},
  publisher={Elsevier}
}

@inproceedings{ralaivaosaona2015repeated,
  title={Repeated fringe subtrees in random rooted trees},
  author={Ralaivaosaona, Dimbinaina and Wagner, Stephan},
  booktitle={Proceedings of the Twelfth Workshop on Analytic Algorithmics and Combinatorics (ANALCO 2015)},
  pages={78--88},
  year={2015},
  publisher={SIAM}
}

@article{seelbach2022distinct,
  title={Distinct fringe subtrees in random trees},
  author={Seelbach Benkner, Louisa and Wagner, Stephan},
  journal={Algorithmica},
  volume={84},
  number={12},
  pages={3686--3728},
  year={2022},
  publisher={Springer}
}

@inproceedings{seelbach2020collection,
  title={On the collection of fringe subtrees in random binary trees},
  author={Seelbach Benkner, Louisa and Wagner, Stephan},
  booktitle={LATIN 2020: Theoretical Informatics},
  series={Lecture Notes in Computer Science},
  volume={12118},
  pages={546--558},
  year={2020},
  publisher={Springer}
}

@inproceedings{wagner2024distinct,
  title={On the number of distinct fringe subtrees in binary search trees},
  author={Wagner, Stephan},
  booktitle={35th International Conference on Probabilistic, Combinatorial and Asymptotic Methods for the Analysis of Algorithms (AofA 2024)},
  series={LIPIcs},
  volume={302},
  pages={13:1--13:15},
  year={2024},
  publisher={Schloss Dagstuhl -- Leibniz-Zentrum f\"ur Informatik}
}

\end{document}